\documentclass[11pt]{article}
\usepackage{latexsym}
\usepackage{amsmath,lipsum}
\usepackage{amssymb,amsfonts,amscd,amsthm}
\usepackage{graphicx}
\usepackage{multirow}
\usepackage{mathabx} 
\usepackage{hyperref}
\usepackage{color}
\usepackage{subcaption}
\usepackage{cite}
\usepackage[boxed]{algorithm2e}
\usepackage[dvipsnames]{xcolor}
\usepackage{changes}
\usepackage{hhline}
\usepackage{tabularx}
\usepackage{array}
\usepackage{siunitx}

\newcommand{\rk}{{\rm rk}}

\usepackage{diagbox}
\usepackage{tikz}

\newtheorem{theorem}{Theorem}[section]
\newtheorem{proposition}[theorem]{Proposition}

\newtheorem{lemma}[theorem]{Lemma}

\newtheorem{rem}[theorem]{Remark}

\newtheorem{example}[theorem]{Example}
\newtheorem{definition}[theorem]{Definition}
\newtheorem{assumption}[theorem]{Assumption}

\newcommand{\sU}{{\mathcal U}}

\newcommand\bC{{\mathbb C}}

\newcommand\bP{{\mathbb P}}

\newcommand\bR{{\mathbb R}}
\newcommand\bZ{{\mathbb Z}}

\DeclareMathOperator*{\Null}{\rm null}
\DeclareMathOperator*{\sign}{\rm sign}
\DeclareMathOperator*{\codim}{\rm codim}

\newcommand{\Xsing}{X_{\rm sing}}
\newcommand{\Xsmooth}{X_{\rm reg}}

\begin{document}
\title{Smooth connectivity in real algebraic varieties}

\author{
Joseph Cummings \thanks{Department of Applied and Computational Mathematics
and Statistics, University of Notre Dame, Notre Dame, IN 46556 (jcummin7@nd.edu)}
\and
Jonathan D. Hauenstein \thanks{Department of Applied and Computational Mathematics
and Statistics, University of Notre Dame, Notre Dame, IN 46556 (hauenstein@nd.edu, \url{https://www.nd.edu/\~jhauenst})}
\and Hoon Hong\thanks{Department of Mathematics,
North Carolina State University, Raleigh, NC 27695 (hong@ncsu.edu, \url{https://hong.math.ncsu.edu/})}
\and Clifford D. Smyth\thanks{Department of Mathematics and Statistics, University of North Carolina
at Greensboro, Greensboro, NC 27402 (cdsmyth@uncg.edu,
\url{https://sites.google.com/view/cliffordsmyth/})}
}

\date{\today}

\maketitle

\begin{abstract}
\noindent
A standard question in real algebraic geometry is to compute 
the number of connected components of a real algebraic variety in affine space.
By adapting an approach for determining connectivity in complements
of real hypersurfaces by Hong, Rohal, Safey El Din, and Schost,
algorithms are presented for computing the number of connected components,
the Euler characteristic, and deciding the connectivity between two points
for a smooth manifold arising as the complement of a real hypersurface 
of a real algebraic variety.
When taking such real hypersurface to be the set of singular points,
this yields an approach for determining smooth connectivity
in a real algebraic variety.  The method is based upon 
gradient ascent/descent paths on the real algebraic variety
and several examples are included to demonstrate~the~approach.


\end{abstract}

\section{Introduction}

In real affine space $\bR^n$, a real algebraic variety has the form
\begin{equation}\label{eq:X}
X = V_\bR(g_1,\dots,g_k) = \{x\in\bR^n~|~g_1(x) = \cdots = g_k(x) = 0\}
\end{equation}
where $g_1,\dots,g_k\in\bR[x_1,\dots,x_n]$, that is, polynomials in $x=(x_1,\dots,x_n)$
with real coefficients.  Many problems in science and engineering can
be translated into questions regarding real algebraic varieties.
For example, the real algebraic variety $X$ could describe the configuration space of 
a mechanism and path planning (e.g., see \cite{BabyGiant,Canny88,OutputModeSwitching})
corresponds with determining connected
paths between two points on $X$.  Moreover, singularity-free path planning 
(e.g., see \cite{SingularityFree1,SingularityFree2})
corresponds with determining smooth connections between two points on $X$.  
Due to its ubiquity, there are many algorithms proposed 
for deciding connectivity with a non-exhaustive list being \cite{RoadmapsVariety,DivideConquer,CCsemi,Canny93,CountingSemi,SingleI,SingleII,SSH87}.

The approach in \cite{HongConnectivity,ConnectivitySemiAlg} considers 
connectivity in $\bR^n\setminus V_\bR(f)$, where $f\in\bR[x_1,\dots,x_n]$,
using connections between critical points via gradient ascent paths.  
The algorithms described below are based on this work, but generalized to consider
$X_f = X\setminus V_\bR(f)$ which is assumed to be smooth.  That is, 
$X\cap V_\bR(f)$ is assumed to at least contain the singular points of $X$.  
Thus, connected components of $X_f$ are smoothly connected in $X$.
For example, Figure~\ref{fig:Illustrative}(a) shows that the pair of intersecting
lines $V_\bR(x_1^2-x_2^2)$ has four smoothly connected components using $f = 4(x_1^2+x_2^2)$ 
while Figure~\ref{fig:Illustrative}(b) shows that the Whitney umbrella
$V_\bR(x_1^2-x_2^2x_3)$ has two smoothly connected components using $f = 4x_1^2+4x_2^2x_3^2+x_2^4$.
Additionally, given a point in $X_f$, one can decide which smoothly connected component
the point belongs to yielding an approach to decide if two
points lie on the same smoothly connected component.  

\begin{figure}
    \centering
    \begin{tabular}{ccc}
    \includegraphics[scale=0.25]{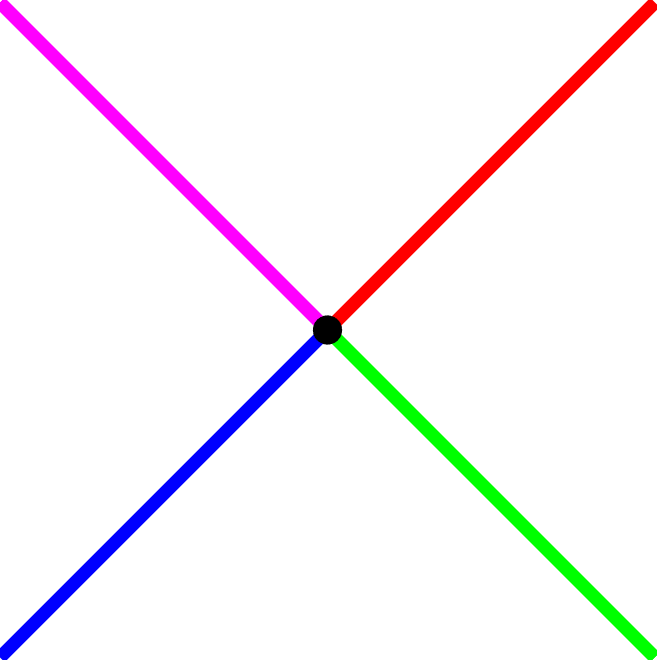} & ~~~~~ &\includegraphics[scale=0.3]{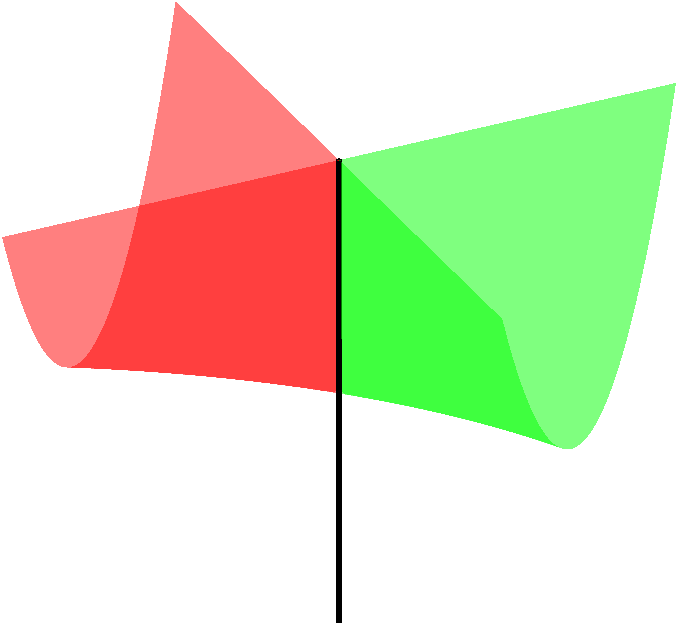} \\
    (a) & & (b) 
    \end{tabular}
    \caption{(a) Four smoothly connected components for a pair of intersecting lines with the singular point at the intersection.  (b) Two smoothly connected components for the Whitney umbrella with the singular points forming the ``handle'' of the umbrella.}
    \label{fig:Illustrative}
\end{figure}

The rest of the paper is organized as follows.  Section~\ref{sec:Preliminaries}
provides a summary of preliminary topics.
Section~\ref{sec:Routing} describes routing functions which
are the basis for the algorithms in Section~\ref{sec:Algorithms}.
Section~\ref{sec:Proofs} proves the correctness of the algorithms
while Section~\ref{sec:Examples} demonstrates the algorithms on some examples.
A short conclusion is provided in Section~\ref{sec:Conclusion}.

\section{Preliminaries}\label{sec:Preliminaries}

The following summarizes some background information that will be used
throughout.

\subsection{Smooth points}\label{sec:SmoothPoints}

The following provides a short summary of smooth points
with more details found in, e.g., \cite{BasuPollackRoyBook}.
For~$X$ as in \eqref{eq:X}, the dimension of $X$, denoted $\dim X$, 
is the largest $d$ such that $(0,1)^d$ injects by a semi-algebraic map into $X$.  
If $\dim X = d$, a point $x\in X$ is a smooth point of $X$, if 
there is an open neighborhood of~$x$ in $X$ 
which is a $d$-dimensional submanifold, i.e., 
the tangent space of $x$ with respect to $X$, denoted $T_x X$, is $d$-dimensional.
A point $x\in X$ is a singular point of $X$ if it is not a smooth point.
Let $\Xsmooth$ and $\Xsing$ be the set of smooth and singular points of~$X$, respectively.
Hence, $X = \Xsmooth \cup \Xsing$ and $\Xsmooth \cap \Xsing = \emptyset$, i.e.,
$\Xsmooth = X\setminus\Xsing$.

\begin{example}\label{ex:LinesWhitney}
For $X=V_\bR(x_1^2-x_2^2)$ as in Figure~\ref{fig:Illustrative}(a), $\dim X = 1$
and $\Xsing = V_\bR(x_1,x_2) = \{(0,0)\}$.
For $X=V_\bR(x_1^2-x_2^2x_3)$ as in Figure~\ref{fig:Illustrative}(b), $\dim X = 2$ and
\mbox{$\Xsing = V_\bR(x_1,x_2) = \{(0,0,x_3)~|~x_3\in\bR\}$}, called the ``handle''
of the Whitney umbrella.
\end{example}

The following will be assumed throughout which relates the 
smooth points with respect to $X$ with the null space of the
Jacobian matrix of the defining polynomials $g_1,\dots,g_k$.  

\begin{assumption}\label{assume:NullSpace}
For $X$ as in \eqref{eq:X}, let $g = \{g_1,\dots,g_k\}$ 
and $Jg(x)$ be the $k\times n$ Jacobian matrix of~$g$ at $x$.  
Then, it is assumed that the system $g$ is such that the following holds:
$$\Xsmooth = \{x\in X~|~\dim \Null Jg(x) = \dim X\}.$$
For $m=n-\dim X$, let $M_m$ be the set of $m\times m$ minors of $Jg(x)$
and $S = \sum_{p\in M_m} p^2$.  
Therefore,
$$\Xsing = X \cap V_\bR(M_m) = X\cap V_\bR(S) = V_\bR(g_1,\dots,g_k,S)
\hbox{~~and~~} \Xsmooth = X\setminus V_\bR(S).$$
\end{assumption}

For example, Assumption~\ref{assume:NullSpace} can always be satisfied
by replacing $g_1,\dots,g_k$ with a generating set for the real radical, e.g., see \cite{RealRadicals}, of 
the ideal $\langle g_1,\dots,g_k\rangle$.

\begin{example}\label{ex:LinesWhitney2}
For $X=V_\bR(x_1^2-x_2^2)$, $g=x_1^2-x_2^2$ satisfies Assumption~\ref{assume:NullSpace}
with $n = 2$ and \mbox{$\dim X = 1$}.  Moreover, $M_1 = \{2x_1,-2x_2\}$ with $S = 4(x_1^2+x_2^2)$ 
yields
$$\Xsing = X\cap V_\bR(M_1) = X\cap V_\bR(S) = V_\bR(x_1,x_2).$$
For $X=V_\bR(x_1^2-x_2^2x_3)$, $g=x_1^2-x_2^2x_3$ satisfies Assumption~\ref{assume:NullSpace}
with $n = 3$ and $\dim X = 2$.  Moreover, $M_1 = \{2x_1,-2x_2x_3,-x_2^2\}$ with 
$S = 4x_1^2 + 4x_2^2x_3^2 + x_2^4$ yields
$$\Xsing = X\cap V_\bR(M_1) = X\cap V_\bR(S) = V_\bR(x_1,x_2).$$ 
In particular, $\{x_1,x_2\}$ is a generating set for the real radical of $\langle g, S\rangle$ in both of these cases.
\end{example}

A semi-algebraic connected set $C\subset X$ is said to be smoothly connected
if $C\subset\Xsmooth$.  Moreover, the smoothly connected components of $X$
are the connected components of $\Xsmooth$.

\begin{example}\label{ex:LinesWhitney3}
For $X = V_\bR(x_1^2-x_2^2)$, the smoothly connected components are
$$C_1 = \{(t,t)~|~t>0)\},\,
C_2 = \{(t,-t)~|~t>0\},\,
C_3 = \{(-t,t)~|~t>0\},\hbox{and\,\,}
C_4 = \{(-t,-t)~|~t>0\}
$$
as illustrated in Figure~\ref{fig:Illustrative}(a).
For $X=V_\bR(x_1^2-x_2^2x_3)$, the smoothly connected components are
$$C_1 = \{(uv,u,v^2)~|~u>0, v\in\bR\}\hbox{\,and\,\,}
C_2 = \{(uv,u,v^2)~|~u<0, v\in\bR\}$$
as illustrated in Figure~\ref{fig:Illustrative}(b).
\end{example}

\subsection{Gradient system}\label{sec:Gradients}

Let $X$ be as in \eqref{eq:X} which satisfies Assumption~\ref{assume:NullSpace}.
Suppose that $a,b\in\bR[x_1,\dots,x_n]$ such that $X\cap V_\bR(b) = \emptyset$
and $f = a/b$.  Hence, $f$ is a rational function defined everywhere on $X$
with $V_\bR(f) = V_\bR(a)$.  
Moreover, suppose that 
\begin{equation}\label{eq:GradientAssumption}
X\setminus V_\bR(f)\subset \Xsmooth \hbox{~~and, equivalently,~~}
\Xsing\subset X\cap V_\bR(f).
\end{equation}
Let $X_f = X\setminus V_\bR(f)\subset \Xsmooth$
Since $X_f$ is a manifold,
the gradient of $f$ on $X_f$, denoted $\nabla_{X_f} f$, is well-defined.
Note that $f\neq0$ on $X_f$ so the connected components of $X_f$
are the union of the connected components of $X_f\cap\{f>0\}$
and $X_f\cap\{f<0\}$.  Thus, one can perform gradient ascent on $X_f$ 
when starting at a point with $f>0$
and gradient descent on $X_f$ when starting at a point with $f<0$
and remain in $X_f$.
In particular, suppose that $x_0\in X_f$ 
and $\sigma_0 = \sign f(x_0) \in \{-1,+1\}$, then the gradient system under consideration
is formally written~as
\begin{equation}\label{eq:GradientSystem}
    \begin{array}{rcll}
    \dot{x}(t) &=& \sigma_0\cdot\nabla_{X_f} f(x) & \hbox{~~on~$X_f$} \\
    x(0) &=& x_0.
    \end{array}
\end{equation}
Of course, gradient systems on manifolds are well-studied, e.g., \cite{GradientProjection,Yang07}.
Computationally, one can consider \eqref{eq:GradientSystem} using a local tangential parameterization.
Suppose that $d=\dim X$ and $x\in X_f$.  
Let $V_x\in\bR^{n\times d}$ be an orthogonal matrix such that its columns form 
an orthonormal basis for $T_x X_f$.
Hence, $Jg(x) V_x = 0$ and $V_x^T V_x = I_d$, the $d\times d$ identity matrix.  
Let $\pi_x:X_f\rightarrow\bR^d$ such that $\pi_x(y) = V_x^T(y-x)$
is the orthogonal projection from~$X_f$ to $T_x X_f$ centered at $x$ (see Figure~\ref{fig:tangential_param}), 
i.e., $\pi_x(x) = 0$.
Since $x\in X_f\subset\Xsmooth$, there exists $\epsilon_x>0$
such that $\pi_x$ restricted to $X_f\cap B_{\epsilon_x}(x)$ is invertible,
where 
$$B_{\epsilon_x}(x) = \{y\in\bR^n~|~\|y-x\|<\epsilon_x\}
\hbox{~~with~~} \|y-x\| = \sqrt{(y_1-x_1)^2+\cdots+(y_n-x_n)^2}.
$$
Since $U_x = \pi_x(X_f\cap B_{\epsilon_x}(x))\subset\bR^d$ is an open neighborhood of the origin,
\eqref{eq:GradientSystem} can be considered locally in parameterizing coordinates $p\in U_x$
with $y(p) = \pi^{-1}(p)\in X_f\cap B_{\epsilon_x}(x)$.  In particular,
\begin{equation}\label{eq:Gradient}
\nabla_{X_f} f(x) = \nabla_{\bR^n} f(x)\cdot V_x \in \bR^d.
\end{equation}
Moreover, the corresponding Hessian matrix is
$$H_{X_f}f(x) = \sum_{i=1}^n \frac{\partial f}{\partial x_i}(x)\cdot W_x^i + 
V_x^T\cdot H_{\bR^n} f(x)\cdot V_x \in \bR^{d\times d}$$
where $W_x^1,\dots,W_x^n\in\bR^{d\times d}$ are symmetric and satisfy the well-constrained linear system
$$
\begin{array}{rcll}
\displaystyle\sum_{i=1}^n \frac{\partial g_j}{\partial x_i}(x)\cdot W_x^i + V_x^T\cdot H_{\bR^n}g_j(x)\cdot V_x &=& 0 & \hbox{for~} j=1,\dots,k, \\
\displaystyle\sum_{i=1}^n (V_x)_{ij}\cdot W_x^i &=& 0 &\hbox{for~} j = 1,\dots,d.
\end{array}
$$
In particular,
\begin{equation}\label{eq:ypExpansion}
y(p) = x + V_x\cdot p + \frac{1}{2}\left[\begin{array}{c} 
p^T\cdot W_x^1\cdot p \\ \vdots \\ p^T\cdot W_x^n\cdot p \end{array}\right] + \hbox{higher order terms}
\end{equation}
such that $g(y(p)) = 0$ since $y(p)\in X_f$.

\begin{figure}[!t]
    \centering
    \includegraphics[scale = 0.2]{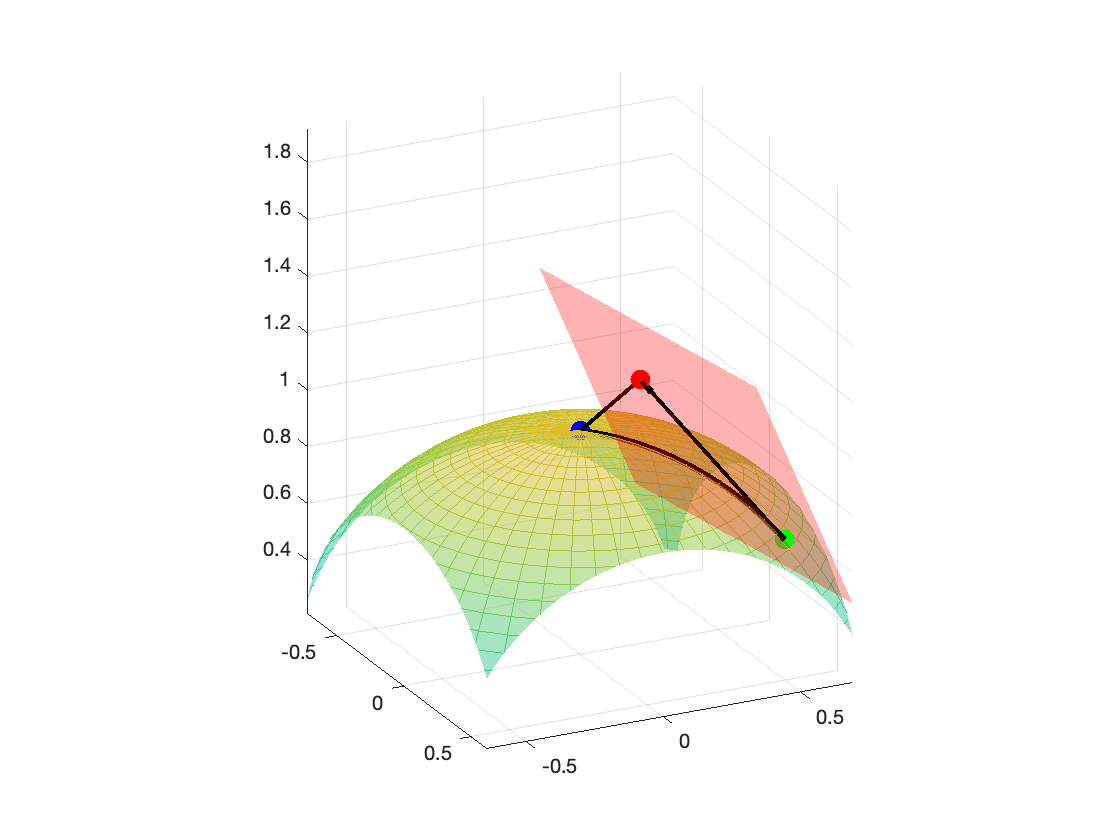}
    \caption{An example of a tangential parameterization of $X$ centered at $x$. The green point is $x$, the red point is $p = x + \mu v$ for $v \in T_x X$, and the blue point is $y = \pi_x^{-1}(p)$.}
    \label{fig:tangential_param}
\end{figure}

\begin{example}
For $X=V_\bR(x_1^2-x_2^2)$ and $f=4(x_1^2+x_2^2)$, let $x=(1,1)$ so that
\begin{itemize}
\item $V_x = \frac{1}{\sqrt{2}}\left[\begin{array}{c} 1 \\ 1 \end{array}\right]$, $W_x^1 = [0]$, $W_x^2 = [0]$;
\item $\nabla_{X_f} f(x) = \nabla_{\bR^2} f(x)\cdot V_x = 8\sqrt{2}$;
\item $H_{X_f} f(x) = \frac{\partial f}{\partial x_1}(x)\cdot W_x^1 + \frac{\partial f}{\partial x_2}(x)\cdot W_x^2 + V_x^T\cdot H_{\bR^2} f(x)\cdot V_x = 8$.
\end{itemize}
Since $X$ is locally linear at $x$, $y(p) = x + V_x\cdot p$
for $p\in\bR$.  Hence, for $q(p) = f(y(p))$,
one has $\nabla_{\bR} q(0) = \nabla_{X_f} f(1,1)$
and $H_{\bR} q(0) = H_{X_f} f(1,1)$.

\medskip
Similarly, for $X = V_\bR(x_1^2-x_2^2x_3)$ and $f=4x_1^2 + 4x_2^2x_3^2 + x_2^4$, let $x=(1,1,1)$ so that
\begin{itemize}
\item $V_x = \frac{1}{\sqrt{2}}\left[\begin{array}{rr} 1 & 1/3\\ 1 & -1/3 \\ 0 & 4/3 \end{array}\right]$,
$W_x^1 = \left[\begin{array}{cc} 0 & 4/27 \\ 4/27 & -16/81\end{array}\right]$,
$W_x^2 = -W_x^1$,
$W_x^3 = -W_x^1/2$;
\item $\nabla_{X_f} f(x) = \nabla_{\bR^3} f(x)\cdot V_x = \frac{2\sqrt{2}}{3}\left[\begin{array}{cc} 15 & 7\end{array}\right]$;
\item $
H_{X_f} f(x) = 
\frac{\partial f}{\partial x_1}(x)\cdot W_x^1 + 
\frac{\partial f}{\partial x_2}(x)\cdot W_x^2 + 
\frac{\partial f}{\partial x_3}(x)\cdot W_x^3 + 
V_x^T\cdot H_{\bR^3} f(x)\cdot V_x = \frac{2}{81}\left[\begin{array}{cc}
567 & 303 \\ 303 & 127 \end{array}\right]$.
\end{itemize}
For $y(p)$ as in \eqref{eq:ypExpansion}
and $q(p) = f(y(p))$, then 
$\nabla_{\bR^2} q(0) = \nabla_{X_f} f(1,1,1)$
and $H_{\bR^2} q(0) = H_{X_f} f(1,1,1)$.
\end{example}

\section{Routing points and routing functions}\label{sec:Routing}

The keys to the algorithms in \cite{HongConnectivity,ConnectivitySemiAlg} 
are routing points and routing functions.  Suppose that $X$
as in \eqref{eq:X} satisfies Assumption~\ref{assume:NullSpace}
and let $r:X\rightarrow\bR$ be a twice continuously differentiable function 
on $X$ such that $r(x) = 0$ for all $x\in\Xsing$, i.e., 
$X_r = X\setminus V_\bR(r) \subset \Xsmooth$.
A point $z\in X$ is called a routing point of $r$ on $X$ 
if $z\in X_r$, i.e., $r(z)\neq0$, and $\nabla_{X_r} r(z) = 0$,
which is equivalent to 
\begin{equation}\label{eq:CriticalPoint}
\dim \Null \left[\begin{array}{cccc}
\nabla_{\bR^n} r(z)^T & \nabla_{\bR^n} g_1(z)^T & \cdots & \nabla_{\bR^n} g_k(z)^T
\end{array}\right] = \dim X.
\end{equation}
One can formulate \eqref{eq:CriticalPoint} using \cite{RankDef}
which, when $k=n-\dim X$, is equivalent to 
using Lagrange multipliers.  
Moreover, a routing point $z$ is nondegenerate if $H_{X_r} r(z)$ is invertible.
Since the gradient system~\eqref{eq:GradientSystem} depends upon
the sign, the index of a nondegenerate routing point is also sign dependent.
In particular, the index of a nondegenerate routing point $z$ 
is the number of eigenvalues of $H_{X_r} r(z)$ of the same sign as $r(z)$.  Eigenvectors of $H_{X_r}r(z)$ corresponding
with the eigenvalues of the same sign as $r(z)$ are
called unstable eigenvector directions.  
For example, if $r(z)>0$, then, since \eqref{eq:GradientSystem} 
is aiming to increase the function value, the index is the dimension of the 
local ascending manifold at~$z$ which is the number of positive eigenvalues of $H_{X_r} r(z)$
and the eigenvectors corresponding with a positive eigenvalue
are the unstable eigenvector directions.

\begin{definition}
    The function $r$ is called a routing function on $X$ if the following conditions hold:
    \begin{enumerate}
            \item $X_r = X\setminus V_\bR(r) \subset \Xsmooth$,
        \item for all $\epsilon>0$, there exists $\delta>0$ such that if $x\in X$
        with $\|x\|\geq \delta$, then $|r(x)|\leq \epsilon$,
        \item the corresponding set of routing points on $X$ is finite and each is nondegenerate, 
        \item for each $\alpha\in\bR\setminus\{0\}$, there
        is at most one routing point $x$ on $X$ 
        satisfying $r(x)=\alpha$, and
        \item the norms of $r$, $\nabla_{X_r} r$, and $H_{X_r} r$ are bounded on $X_r$.
    \end{enumerate}
\end{definition}

In particular, a routing function vanishes on $\Xsing$ as well as at infinity, and each level curve
contains at most one routing point.
Therefore, if $C$ is a connected component of $X_r$, then $r$ on $C$ must obtain either a minimum
(if $r<0$ on $C$) or a maximum (if $r>0$ on $C$), which must occur at a routing point.  
The following formalizes this.

\begin{proposition}\label{Prop:ConnectedCompRoutingPoints}
With the assumptions and definitions above, there is at least one routing point
in each connected component of $X_r$ of index $0$.
\end{proposition}

\begin{example}\label{ex:LinesRouting}
For $X=V_\bR(x_1^2-x_2^2)$, the function $f(x) = 4(x_1^2 + x_2^2)$ is not a routing function
on~$X$ since $f$ is unbounded on $X$.  Nonetheless, consider the
following rational function related to~$f$:
\begin{equation}\label{eq:RinExample}
r(x) = \frac{4(x_1^2+x_2^2)}{((x_1-1/2)^2+(x_2-1/3)^2 + 1)^2},
\end{equation}
so that $X_r = X_f$.  Hence, $r$ is a routing function with four routing points $(\pm 7/(6\sqrt{2}),\pm 7/(6\sqrt{2}))$.
Proposition~\ref{Prop:ConnectedCompRoutingPoints} holds with exactly one point in
each connected component of $X_r$ as illustrated in Figure~\ref{fig:LinesRouting}.
It is easy to verify that $r$ takes different values
at each of the routing points, which corresponds
with a local maximum of $r$ along $X_r$, 
i.e., each has index $0$.

\begin{figure}[!ht]
    \centering
    \includegraphics[scale=0.2]{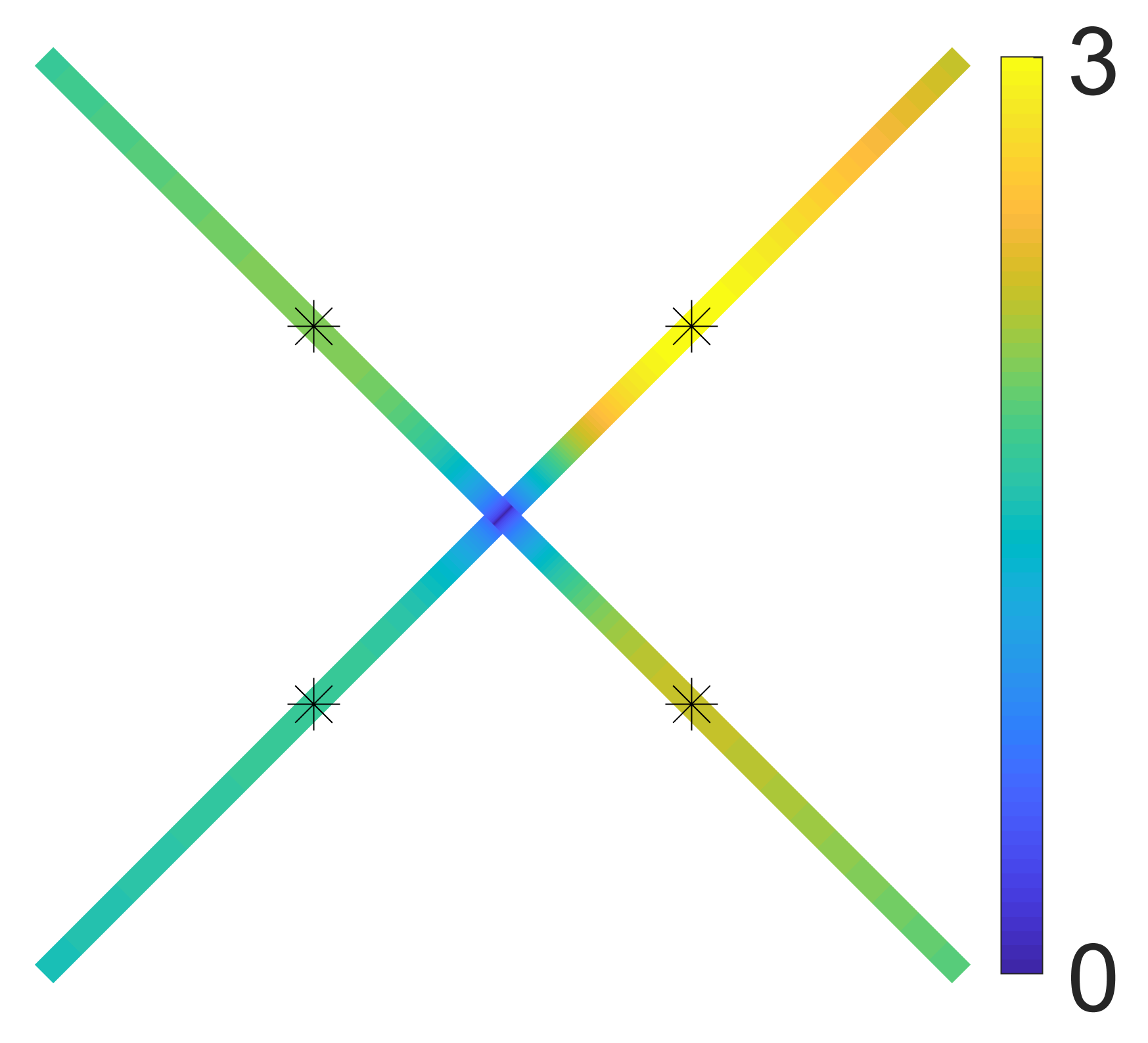}
    \caption{Pair of intersecting lines colored based on the value of the routing function 
    with the four routing points marked, each corresponding to a local maxima of the routing function.}
    \label{fig:LinesRouting}
\end{figure}
\end{example}

The following provides a generalization of the 
construction used in Example~\ref{ex:LinesRouting} 
to create a routing function derived from \cite{HongConnectivity,ConnectivitySemiAlg}.

\begin{theorem}\label{thm:RoutingFunction}
Suppose that $X$ as in \eqref{eq:X}
satisfies Assumption~\ref{assume:NullSpace} 
and $f\in\bR[x_1,\dots,x_n]$ such that 
$X_f = X\setminus V_\bR(f)\subset\Xsmooth$.  
Let $\ell\in\bZ_{>0}$ such that $2\ell > \deg f$.
Then, there is a Zariski open dense subset $\sU\subset\bR^n$
such that, for every $c\in\sU$,
$$r_c(x) = \frac{f(x)}{((x_1-c_1)^2+\cdots+(x_n-c_n)^2+1)^\ell}$$
is a routing function on $X$.
\end{theorem}
\begin{proof}
Let $c\in\bR^n$ and define $D_c(x) = (x_1-c_1)^2+\cdots+(x_n-c_n)^2+1$.  Thus, $r_c(x) = f(x)\cdot D_c(x)^{-\ell}$.
Since $D_c$ does not vanish on $\bR^n$,
$r_c$ is infinitely differentiable
with $X_{r_c} = X_f \subset \Xsmooth$.
Moreover, $2\ell > \deg f$
ensures that $r_c$ is bounded on $\bR^n$ (and hence on $X$)
and vanishes at infinity.  
Additionally, 
\begin{equation}\label{eq:GradientR}
\nabla_{\bR^n} r_c(x) = \nabla_{\bR^n} f(x)\cdot D_c(x)^{-\ell}
-2\cdot \ell\cdot f(x)\cdot D_c(x)^{-\ell-1}\cdot (x-c)
\end{equation}
is also bounded on $\bR^n$ for the same reason.
Since $V_x$ in \eqref{eq:Gradient} is orthonormal,
this shows that $\nabla_{X_{r_c}} r_c$ is bounded on $X$.
A similar argument shows that $H_{X_{r_c}} r_c$ is also
bounded on $X_{r_c}$ since the matrices~$W_x^j$ 
describe the local curvature of $X_f \subset \Xsmooth$ and $V_x$ is orthonormal.

All that remains 
is showing that the set of routing points is finite,
each is nondegenerate,
and evaluate to different values of $r_c$
for values $c$ in 
a Zariski open dense subset of $\bR^n$.
Since $f$ and $D_c$ are nonzero on $X_{r_c}$, 
\eqref{eq:CriticalPoint} can be
equivalently formulated as
$$
\dim \Null
\left[
\begin{array}{cccc}
\frac{1}{2\cdot\ell}\cdot 
\frac{\nabla_{\bR^n} f(z)^T}{f(z)} - \frac{(z - c)}{D_c(z)}
& \nabla_{\bR^n} g_1(z)^T & \cdots & 
\nabla_{\bR^n} g_k(z)^T
\end{array}
\right] = \dim X.
$$
Since the last $k$ columns has a null space equal
to $\dim X$ on $X_{r_c}=X_f\subset\Xsmooth$ via Assumption~\ref{assume:NullSpace}, 
this shows that adjusting the value of $c$
will change both the location 
and value of $r$ of routing points $z$
to ensure that the first column
is contained in the span of the last $k$ columns.
Therefore, the result follows from 
an algebraic version of Sard's theorem, e.g., see \cite[Thm.~A.4.10]{SW05}.
\end{proof}

Theorem~\ref{thm:RoutingFunction} shows that 
one can obtain a routing function using a generic $c\in\bR^n$.

\begin{example}\label{ex:Circle}
To demonstrate a value of $c$ that does not work,
consider $X = V_\bR(x_1^2+x_2^2-1)$ with $f=4(x_1^2+x_2^2)$.
For $\ell = 2$, consider $c=0$ so that $r_0(x)$
is as in \eqref{eq:RinExample}.  With this, every point 
on $X$ is a routing point so that $r_0$ is not a routing 
function on $X$.  For, say $c=(1/2,1/3)$, then 
$$r_c(x) = \frac{4(x_1^2+x_2^2)}{((x_1-1/2)^2+(x_2-1/3)^2+1)^2}$$
is a routing function with two routing
points: $(3/\sqrt{13}, 2/\sqrt{13})$
is a maximum and its antipodal point
$(-3/\sqrt{13}, -2/\sqrt{13})$ is a minimum.  
\end{example}

\begin{example}\label{ex:WhitneyRouting}
For $X=V_\bR(x_1^2-x_2^2x_3)$, consider
$f(x) = 4x_1^2+4x_2^2x_3^2+x_2^4$
from Example~\ref{ex:LinesWhitney2}.
Consider $c\in\bR^3$ and $\ell = 3$ so that
$$r_c(x) = \frac{4x_1^2+4x_2^2x_3^2+x_2^4}{((x_1-c_1)^2+(x_2-c_2)^2+(x_3-c_3)^2+1)^3}.$$
When $c=0$, then $r_0$ is not a routing function
since the two routing points $(0,\pm \sqrt{2},0)$
are degenerate, i.e., $H_{X_r} r$ is not invertible
at these two points.  
With, say $c=(1/2,1/3,1/4)$, then 
$r_c$ is a routing function with six routing points:
4 local maxima and 2 saddles with index 1.  
\end{example}

\begin{example}\label{ex:WhitneyAxes}
For $X=V_\bR(x_1^2-x_2^2x_3)$, consider
$f(x) = x_1x_2x_3$ which means that $X_f$
is the set of all points in $X$ where all three
coordinates are nonzero.  The function
$$
r_0(x) = \frac{x_1x_2x_3}{(x_1^2+x_2^2+x_3^2+1)^2}
$$
is not a routing function since
two pairs of routing points are in the same
level set.  However, as suggested by the proof of Theorem~\ref{thm:RoutingFunction},
perturbing away from being centered at the origin
destroys the symmetric structure so that, 
say $c=(1/6,1/5,1/4)$, yields a routing function, namely
\begin{equation}\label{eq:WhitneyR0}
r_c(x) = \frac{x_1x_2x_3}{((x_1-1/6)^2+(x_2-1/5)^2+(x_3-1/4)^2+1)^2}.
\end{equation}
This yields four routing points, two are local
maxima with $r>0$ and two are local minima with
$r<0$.
With the sign dependent notion of index,
all four have index 0.

\begin{figure}[!b]
    \centering
    \includegraphics[scale=0.2]{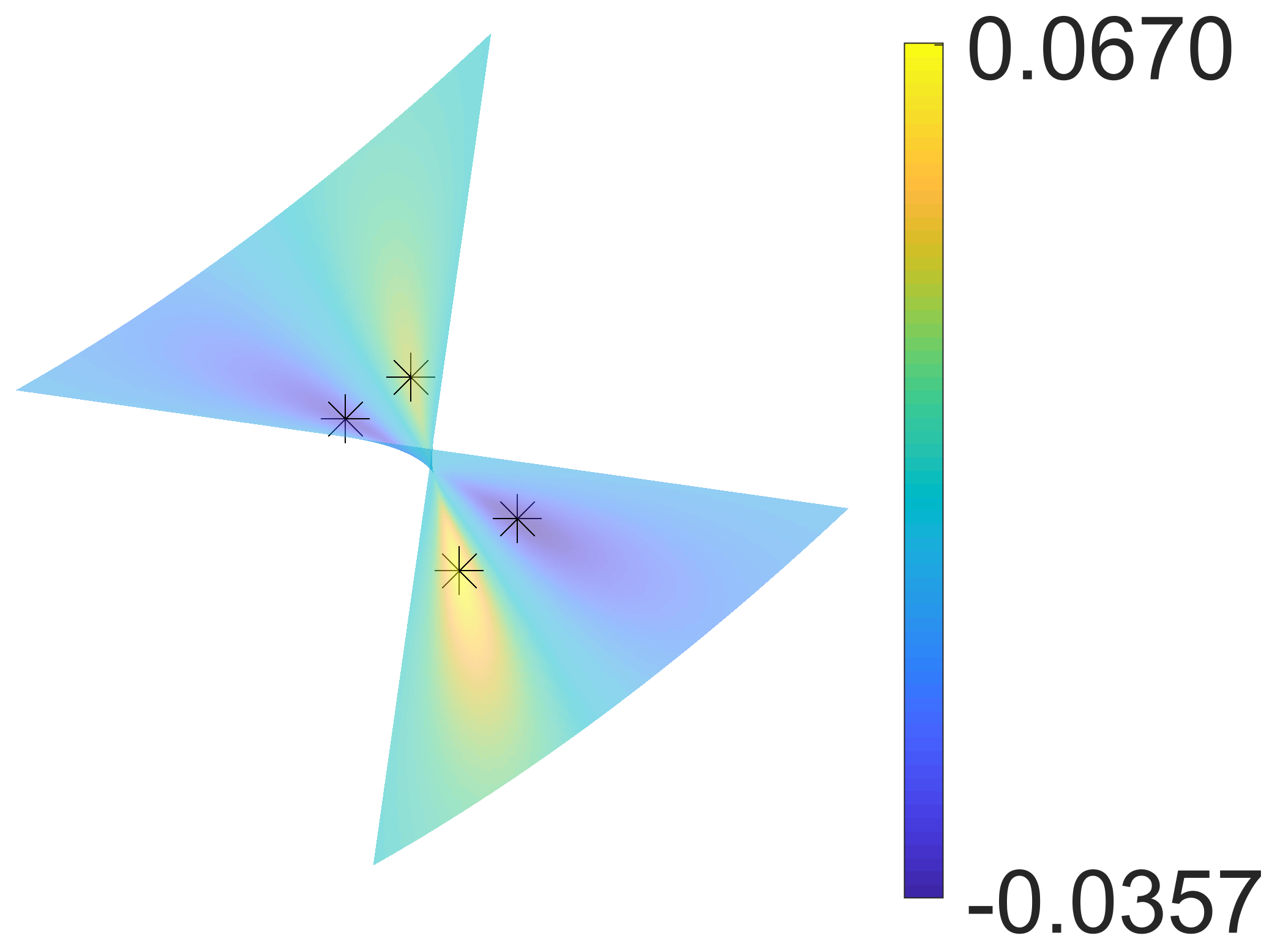}
    \caption{Whitney umbrella with coordinate axes removed
    and four routing points marked, each corresponding to a local optima of $r_c$ as in \eqref{eq:WhitneyR0}.}
    \label{fig:WhitneyAxesRouting}
\end{figure}
\end{example}

For a routing function $r$ on $X$,
the Euler characteristic of $X_r$ 
is determined by counting the number of routing points 
of each index as summarized in the following.

\begin{theorem}\label{thm:Euler}
When $r$ is a routing function for $X$ as in \eqref{eq:X}
which satisfies Assumption~\ref{assume:NullSpace},
then the Euler characteristic of $X_r$ is
\begin{equation}\label{eq:Euler}
    \chi(X_r) = \sum_{j=0}^d (-1)^j \rk_j^{X_r}
\end{equation}
where $\rk_j^{X_r}$ is the number of routing points of $r$ on $X_r$ of index $j$.
\end{theorem}

\begin{proof}
    Let $C_1, \dotsc, C_s$ be the connected components of $X_r$. Since the Euler characteristic is additive, it is enough to prove the formula on each connected component.
    
    Suppose that $C$ is a connected component of $X_r$
    and let $z_1,\dots,z_k$ be the routing points of~$r$ contained 
    in $C$.  
Since $r$ has the same sign on $C$, we will consider
    the positive and negative cases separately. 
    Suppose that $r>0$ on $C$ and let $0 < \delta < \min\{r(z_1),\dots,r(z_k)\}$.    
    Thus, $r$ is a Morse function on~$C$ corresponding to the gradient vector field $\nabla_{X_r}r$
    and we can retract $C$ to $C_{\delta} = (r|_C)^{-1}([\delta, \infty))$ via the vector field, where $r|_C$ is the restriction of $r$ to $C$.  
    Since $C_\delta$ is a smooth compact manifold with boundary as $r|_C$ is bounded, it is well-known, e.g. \cite[\S 2.3]{Liviu}, that 
    \[
        \chi(C_\delta) = \sum_{j=0}^d (-1)^j \rk_j^{C}
    \]
    where $\rk_j^C$ is the number of routing points in $C$ of index $j$ (corresponding to positive eigenvalues). 
    As the Euler characteristic is invariant under homotopy, $\chi(C) = \chi(C_\delta)$ as claimed.

    Similarly, if $r<0$ on $C$, let
    $\max\{r(z_1),\dots,r(z_k)\} < \delta < 0$.
    Thus, $r$ is a Morse function on $C$ corresponding to the gradient vector field $-\nabla_{X_r}r$
    and we can retract $C$ to $C_\delta = (r|_C)^{-1}((-\infty, \delta])$ via the vector field.
    The same formula holds with index
    corresponding to negative eigenvalues.
\end{proof}

\begin{example}\label{ex:Euler}
For Example~\ref{ex:LinesRouting}, the corresponding Euler characteristic is $(-1)^0\cdot 4 = 4$.
Additionally, the Euler characteristic of the unit
circle is $(-1)^0\cdot 1 + (-1)^1\cdot 1 = 0$
from Example~\ref{ex:Circle}.
From Example~\ref{ex:WhitneyRouting}, 
the Euler characteristic of the Whitney umbrella
with the ``handle'' removed is $(-1)^0\cdot 4 + (-1)^1\cdot 2 = 2$.
Finally, the Euler characteristic of the Whitney umbrella
with the coordinate axes removed 
is $(-1)^0\cdot 4 = 4$ from Example~\ref{ex:WhitneyAxes}.
\end{example}

\section{Connectivity algorithms}\label{sec:Algorithms}

The number of connected components is bounded above
by the number of routing points of index~$0$
via Proposition~\ref{Prop:ConnectedCompRoutingPoints}.
The following shows how to partition the set
of routing points into subsets precisely corresponding
to the connected components 
using the gradient system \eqref{eq:GradientSystem}.
The key to this computation is tracking 
from routing points with positive index.
Since a routing point is a stationary point of
\eqref{eq:GradientSystem}, for routing points of positive
index, one needs to consider 
\eqref{eq:GradientSystem} in an instantaneous initial direction
in order to have non-stationary trajectory.
By adapting the approach of
\cite{HongConnectivity,ConnectivitySemiAlg},
this yields a connectivity algorithm.  
Thus, this gradient representation of smoothly connected components
via a routing function and routing points provides
an analog to witness sets 
for complex irreducible varieties which permit 
membership testing, e.g., see~\cite[Chap.~13-15]{SW05}.

To set the stage, first consider an
initial point which is not a routing point.

\begin{proposition}\label{Prop:TrajectoryNotRouting}
Suppose that $X$ in \eqref{eq:X} satisfies
Assumption~\ref{assume:NullSpace}
and $r$ is a routing function on~$X$.  
If $x_0\in X_r$ is not a routing point, i.e., $\nabla_{X_r}r(x_0)\neq 0$, then \eqref{eq:GradientSystem} 
defines a unique trajectory which limits to a routing point of $r$
on $X_r$.    
\end{proposition}
\begin{proof}
Recall that for a routing function, $r$, $\nabla_{X_r} r$, and $H_{X_r}r$ are all 
bounded on $X_r$ and hence,
for any $a>0$,
$X_r\cap r^{-1}((-\infty,-a]\cup[a,\infty))$
is compact.
Standard existence and uniqueness theory 
for initial value problems, e.g., \cite[\S~8.1, Thm.~3]{KincaidCheney}, adapted to manifolds
shows that
\eqref{eq:GradientSystem} has a unique solution $x(t)\in X_r$
for all $t\geq0$.  Moreover,
since $x_0$ is not a routing point, 
$r(x(t))$ is strictly monotonic
while $x(t)$ must be bounded.  
Hence, $z = \lim_{t\rightarrow\infty} x(t)$ is well-defined
with $|r(z)|>|r(x_0)|>0$.  
Boundedness also implies
that one must have $\nabla_{X_r}r(z) = 0$,
i.e., the trajectory limits to a routing point.
\end{proof}

Proposition~\ref{Prop:TrajectoryNotRouting} 
together with stationary trajectories
starting at routing points shows that
the gradient vector field on $X_r$ defined
by $\sign(r(x))\cdot\nabla_{X_r}r(x)$
is complete.  

Next, consider an initial point which is a routing point
$z\in X_r$ with an initial direction $v\in\bR^n$
in the tangent space of $X_r$ at $z$ with $\|v\|=1$ 
Since $z$ is nondegenerate,
there exists $\epsilon_0>0$ such that
$z$ is the unique routing point in
$X_r\cap B_{\epsilon_0}(z)$ and the orthogonal
projection from $X_r$ to $T_z X_r$ centered at $z$
is invertible.  Thus, for $\epsilon\in(0,\epsilon_0)$,
one can apply Proposition~\ref{Prop:TrajectoryNotRouting}
with initial condition $x_0 = \pi^{-1}(z+\epsilon\cdot v)$
to yield trajectory $x_\epsilon(t)$.  
By uniqueness, $\lim_{\epsilon\rightarrow0^+} x_\epsilon(t)$
is well-defined and limits to a routing point of $r$ on
$X_r$.  This is summarized in the following.

\begin{proposition}\label{Prop:TrajectoryRouting}
Suppose that $X$ in \eqref{eq:X} satisfies
Assumption~\ref{assume:NullSpace}
and $r$ is a routing function on~$X$.  
Suppose that $z\in X_r$ is a routing 
point and $v\in\bR^n$ is a unit
vector in the tangent space of $X_r$ at $z$.
Letting $x_\epsilon(t)$
be the trajectory from Proposition~\ref{Prop:TrajectoryNotRouting}
starting at the orthogonal projection of $z+\epsilon\cdot v$
onto $X_r$,
then $x(t) = \lim_{\epsilon\rightarrow0^+} x_\epsilon(t)$
is well-defined trajectory 
which limits to a routing point of $r$~on~$X_r$.
\end{proposition}

\begin{example}\label{ex:WhitneyTrajectories}
To illustrate Propositions~\ref{Prop:TrajectoryNotRouting}
and~\ref{Prop:TrajectoryRouting},
consider Example~\ref{ex:WhitneyRouting}
with $c=(1/2,1/3,1/4)$.  
First, consider the trajectory emanating 
from the non-routing point $x_0 = (-2.25,1.5,2.25)$
which limits to a routing point that is a local maximum. 
In Figure~\ref{fig:WhitneyTrajectoriesRouting},
$x_0$ is shown in red with the trajectory (yellow)
limiting to a routing point (black).  
Second, consider the trajectories emanating
from the index 1 saddle points, 
approximately $(-0.5002, 1.0635, 0.2212)$
and $(-0.5255, -1.3526, 0.1509)$,
in the two directions arising from the unstable eigenvector.
In Figure~\ref{fig:WhitneyTrajectoriesRouting},
each of these two trajectories are shown (green and magenta)
which limit to a routing point that is a local maximum. 

\begin{figure}[!ht]
    \centering
    \includegraphics[scale=0.35]{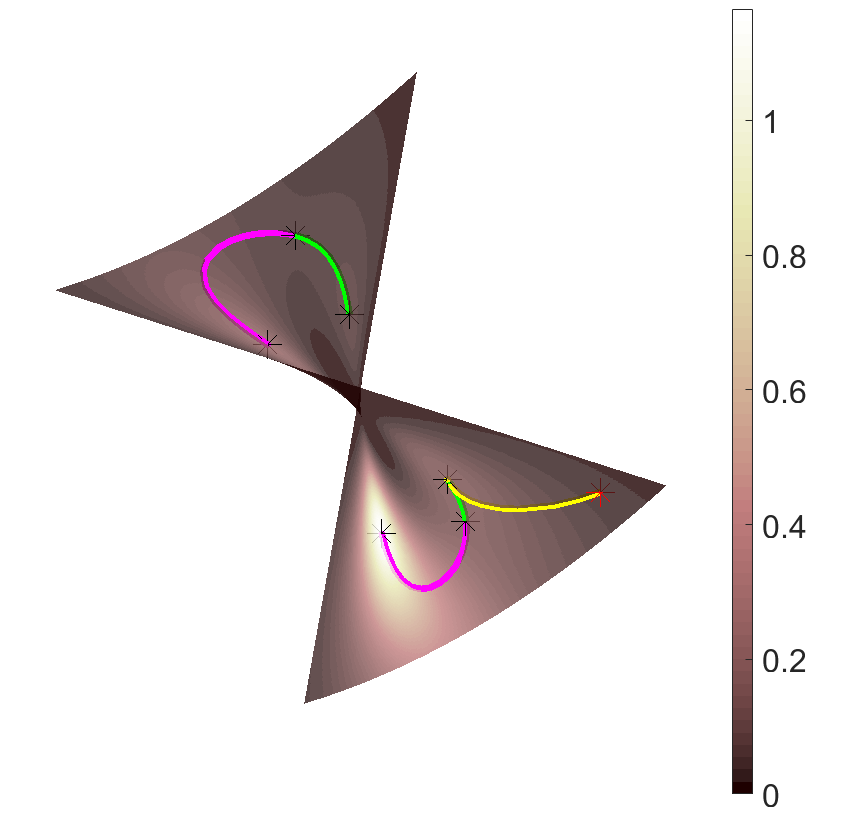}
    \caption{Whitney umbrella with ``handle'' removed
    and six routing points marked (black).
    Illustration of a trajectory (yellow) from a non-routing point (red)
    along with two trajectories (green and magenta) emanating from each index 1 saddle
    in the unstable eigenvector direction
    which connects the local maxima on the same connected component.}
    \label{fig:WhitneyTrajectoriesRouting}
\end{figure}
\end{example}

As observed in Figure~\ref{fig:WhitneyTrajectoriesRouting},
the trajectories from the saddles in an unstable
eigenvector direction connect the local optimum
that lie in the same connected component.  
This holds in general and is summarized
in Algorithm~\ref{alg:Connectivity} and Theorem~\ref{thm:Connectivity}.

\SetKwInput{KwIn}{Input}
\SetKwInput{KwOut}{Output}

\begin{algorithm}[!t]
\caption{Euler characteristic and connected components}  
\label{alg:Connectivity}
\SetAlgoLined
\KwIn{Polynomials $g_1,\dots,g_k\in\bR[x_1,\dots,x_n]$ with $X = V_\bR(g_1,\dots,g_k)$ satisfying Assumption~\ref{assume:NullSpace}
and routing function $r$.}
\KwOut{Euler characteristic of $X_r$ and partitioned subset of the routing points of $r$ on $X_r$
corresponding to the connected components of $X_r$.}

Computing routing points $r$ on $X_r$, say $z_1,\dots,z_m$,
and corresponding indices, say $i_1,\dots,i_m$.

Compute $\chi = \sum_{j=1}^m (-1)^{i_j}$.

Initialize $A = I_m$, the $m\times m$ identity matrix.

\For{$j=1,\dots,m$}{
  \ForEach{unstable eigenvector $v$ for $H_{X_r}r(z_j)$}{ 
    Compute limit routing point from $z_j$
    in the direction $v$ with respect to $r$, say $z_{w_+}$.

    Set $A_{jw_+} = A_{w_+j} = 1$.

    Compute limit routing point from $z_j$
    in the direction $-v$ with respect to $r$, say $z_{w_-}$.

    Set $A_{jw_-} = A_{w_-j} = 1$.
  }
}

Set $M$ to be the transitive closure of $A$.

Partition $\{z_1,\dots,z_m\}$ based on the connected
components of $M$, say $C_1,\dots,C_s$.

\Return{($\chi$,$\{C_1,\dots,C_s\}$)}

\end{algorithm}

The matrix $A$ in Algorithm~\ref{alg:Connectivity} 
is constructed to be reflexive (diagonal
entries are $1$) and symmetric.  
Thus, transitive closure means
to enforce the transitivity property,
i.e., if $z_i$ is connected to $z_j$ which is connected
to $z_k$, then $z_i$ is connected to $z_k$.
Using Boolean matrix multiplication and addition,
the transitive closure of $A\in\bR^{m\times m}$ is 
$$M = A + A^2 + \cdots + A^m.$$
In particular, $M_{ij} = 1$ if and only if 
$z_i$ and $z_j$ lie on the same connected component
of $X_r$.

Algorithm~\ref{alg:ConnectivityQ} uses the output
of Algorithm~\ref{alg:Connectivity} 
as input to answer connectivity queries.  

\begin{algorithm}[!t]
\caption{Connectivity query}  
\label{alg:ConnectivityQ}
\SetAlgoLined
\KwIn{Polynomials $g_1,\dots,g_k\in\bR[x_1,\dots,x_n]$ with $X = V_\bR(g_1,\dots,g_k)$ satisfying Assumption~\ref{assume:NullSpace}, 
routing function $r$, partitioned subsets $C_1,\dots,C_s$ 
of the routing points of $r$ on $X_r$
corresponding to the connected components of $X_r$,
and points $p,q\in X_r$.}
\KwOut{$True$ if $p$ and $q$ lie on the same
connected component of $X_r$ and $False$ otherwise.}

\eIf{$p$ is a routing point}{
  Set $p'\in\{1,\dots,s\}$ so that $p\in C_{p'}$.
}{
  Compute limit routing point $z_{j_p}$ via Proposition~\ref{Prop:TrajectoryNotRouting}
  starting at $p$.

  Set $p'\in\{1,\dots,s\}$ so that $z_{j_p}\in C_{p'}$.
}

\eIf{$q$ is a routing point}{
  Set $q'\in\{1,\dots,s\}$ so that $q\in C_{q'}$.
}{
  Compute limit routing point $z_{j_q}$ via Proposition~\ref{Prop:TrajectoryNotRouting}
  starting at $q$.

  Set $q'\in\{1,\dots,s\}$ so that $z_{j_q}\in C_{q'}$.
}
\Return{$True$ if $p'=q'$, else $False$}

\end{algorithm}

\begin{theorem}\label{thm:Connectivity}
Algorithms~\ref{alg:Connectivity} 
and~\ref{alg:ConnectivityQ} are correct.    
\end{theorem}

A proof is presented in Section~\ref{sec:Proofs}.

\begin{example}
From Example~\ref{ex:WhitneyTrajectories},
one can write the matrices $A$ and $M$ in Algorithm~\ref{alg:Connectivity}
as
$$A = \left[\begin{array}{cccccc}
1 & 0 & 1 & 0 & 0 & 0\\
0 & 1 & 1 & 0 & 0 & 0\\
1 & 1 & 1 & 0 & 0 & 0\\
0 & 0 & 0 & 1 & 0 & 1\\
0 & 0 & 0 & 0 & 1 & 1\\
0 & 0 & 0 & 1 & 1 & 1\\
\end{array}\right]
\hbox{~~~and~~~~}
M = \left[\begin{array}{cccccc}
1 & 1 & 1 & 0 & 0 & 0\\
1 & 1 & 1 & 0 & 0 & 0\\
1 & 1 & 1 & 0 & 0 & 0\\
0 & 0 & 0 & 1 & 1 & 1\\
0 & 0 & 0 & 1 & 1 & 1\\
0 & 0 & 0 & 1 & 1 & 1\\
\end{array}\right].
$$
That is, $M$ shows that $X_r$ has two connected components
each corresponding with three routing points, two
local maxima and a saddle of index 1
as illustrated in Figure~\ref{fig:WhitneyTrajectoriesRouting}.
Thus,
since $X_r = \Xsmooth$ where $X$ is the Whitney umbrella,
the Whitney umbrella has two smoothly
connected components.
\end{example}

\begin{example}
From Example~\ref{ex:WhitneyAxes},
the Whitney umbrella with the coordinate axes
removed decomposes into four connected components.
\end{example}

\section{Correctness proofs}\label{sec:Proofs}

The following presents a proof to Theorem~\ref{thm:Connectivity}.
Clearly, the Euler characteristic follows from Theorem~\ref{thm:Euler}.  The correctness
of Algorithm~\ref{alg:ConnectivityQ} 
follows from the statement and proof of
Proposition~\ref{Prop:TrajectoryNotRouting}
with strict monotonicity showing 
$p$ and $z_{j_p}$ as well as $q$ and $z_{j_q}$
lie on the same connected component of $X_r$.
Hence, the correctness of Algorithm~\ref{alg:ConnectivityQ}
follows from the correctness of Algorithm~\ref{alg:Connectivity}.
Similarly, in Algorithm~\ref{alg:Connectivity}, Proposition~\ref{Prop:TrajectoryRouting}
yields that routing point $z_j$ is connected to 
routing points~$z_{w+}$ and~$z_{w-}$.  
Since the transitive closure ensures
the transitivity of the connections described by~$A$,
the only thing left to show regarding 
Algorithm~\ref{alg:Connectivity}
is that the connections derived from unstable eigenvectors
suffice for determining the connected components
via the mountain pass theorem.

For a routing point $z$, the stable manifold of $z$
with respect to $r$ on $X_r$ is
$$M_r(z) = \{z\}\cup\{x_0\in X_r~|~\nabla_{X_r}r(x_0)\neq0
\hbox{~and trajectory from
Proposition~\ref{Prop:TrajectoryNotRouting} limits to $z$}\}.
$$
The proof of \cite[Thm~4.2]{BanyagaHurtubise2004}
can be trivially adapted to this case with appropriate adjustments
to conclude that $\codim M_r(z)$ is the index of $z$
with respect to $r$, e.g., if $z$ is a routing point with 
$r(z) > 0$ and is a local maximum (index 0), then
$\dim M_r(z) = \dim X_r$.
Clearly, uniqueness of trajectories yields
$$X_r = \bigsqcup_{\hbox{routing points $z$}} M_r(z).$$
Therefore, for a connected component $C$ of $X_r$, one has
\begin{equation}\label{eq:ConnectedRouting}
C = \bigsqcup_{\hbox{routing points $z\in C$}} M_r(z).
\end{equation}
In this way, one is identifying each connected component
with the finitely many routing points that lie inside of it.  

Suppose that $z_0\in C$ is a routing point.  If $\dim M_r(z_0) < \dim X_r$,
i.e., the index of $z_0$ is positive, then select any unstable eigenvector
direction and let $z_1$ be the corresponding routing point 
as in Proposition~\ref{Prop:TrajectoryRouting}.
Hence, $z_1\in C$ and $|r(z_0)| < |r(z_1)|$.  If $\dim M_r(z_1) < \dim X_r$, one can
repeat this process yielding a sequence of routing points
$z_0,z_1,\dots$ with $|r(z_j)| < |r(z_{j+1})|$.  
Hence, this must be a sequence of distinct routing points.  
Since there are only finitely many routing
points, this process must terminate after finitely many steps 
yielding, say, a routing point $z_\ell\in C$ with $\dim M_r(z_\ell) = \dim X_r$,
i.e., the index of $z_\ell$ is 0.  Therefore,
this shows that every routing point in~$C$ is connected to some 
routing point of index 0 in $C$ by trajectories following unstable
eigenvector directions.  Hence, all that remains is showing
connectivity between routing points of index~0~in~$C$.

For a routing point $z$, let $\overline{M_r(z)}$ denote the 
Euclidean closure of $M_r(z)$ in $X_r$.  
Since
$$
\bigsqcup_{\hbox{index $> 0$ routing points $z\in C$}} M_r(z)
$$
has positive codimension, it immediately follows that
$$
C = \bigcup_{\hbox{index 0 routing points $z\in C$ }} \overline{M_r(z)}.
$$
Let $z$ and $z'$ be distinct routing points in $C$ of index $0$
such that $S_{z,z'} = \overline{M_r(z)}\cap\overline{M_r(z')}\neq \emptyset$.  
Hence, $r(z) \neq r(z')$.
Since $r$ vanishes at infinity, for any $\alpha\neq 0$, 
$V_{\bR}(r-\alpha)$ is compact.  Thus, $r$ satisfies
the Palais-Smale condition and the Mountain Pass Theorem \cite{Ambrosetti84} (see 
also \cite[Thm.~3]{MountainPass}) shows the existence of an index 1 routing
point $z''$ in $S_{z,z'}$ so that $z$ and $z'$ are connected by trajectories
emanating from the unstable eigenvector direction at $z''$.  

Since $C$ is connected and there are only finitely many
routing points, one can repeat this argument to create a sequence
of connecting trajectories between any two index $0$ routing points in $C$.
Therefore, Algorithm~\ref{alg:Connectivity}
correctly identifies the connected components of $X_r$.

\section{Examples}\label{sec:Examples}

The following considers various examples for computing
smoothly connected components.  The routing points
were computed using {\tt Bertini}~\cite{BHSW06}
and the trajectories were computed using {\tt Matlab}.

\subsection{Positive solutions}

In order to compute the number of smoothly connected
components in the positive orthant, one can choose
a routing function $r$ so that $X_r\subset (\bR^*)^n$
where $\bR^* = \bR\setminus\{0\}$ and then only
consider routing points in positive orthant.  
For example, one can reconsider
both Example~\ref{ex:LinesRouting} 
and Example~\ref{ex:WhitneyAxes} 
to see that each has a single routing point
in the positive orthant, i.e., each has one
smoothly connected component in the positive orthant.

\begin{figure}[!t]
    \centering
    \includegraphics[scale=0.15]{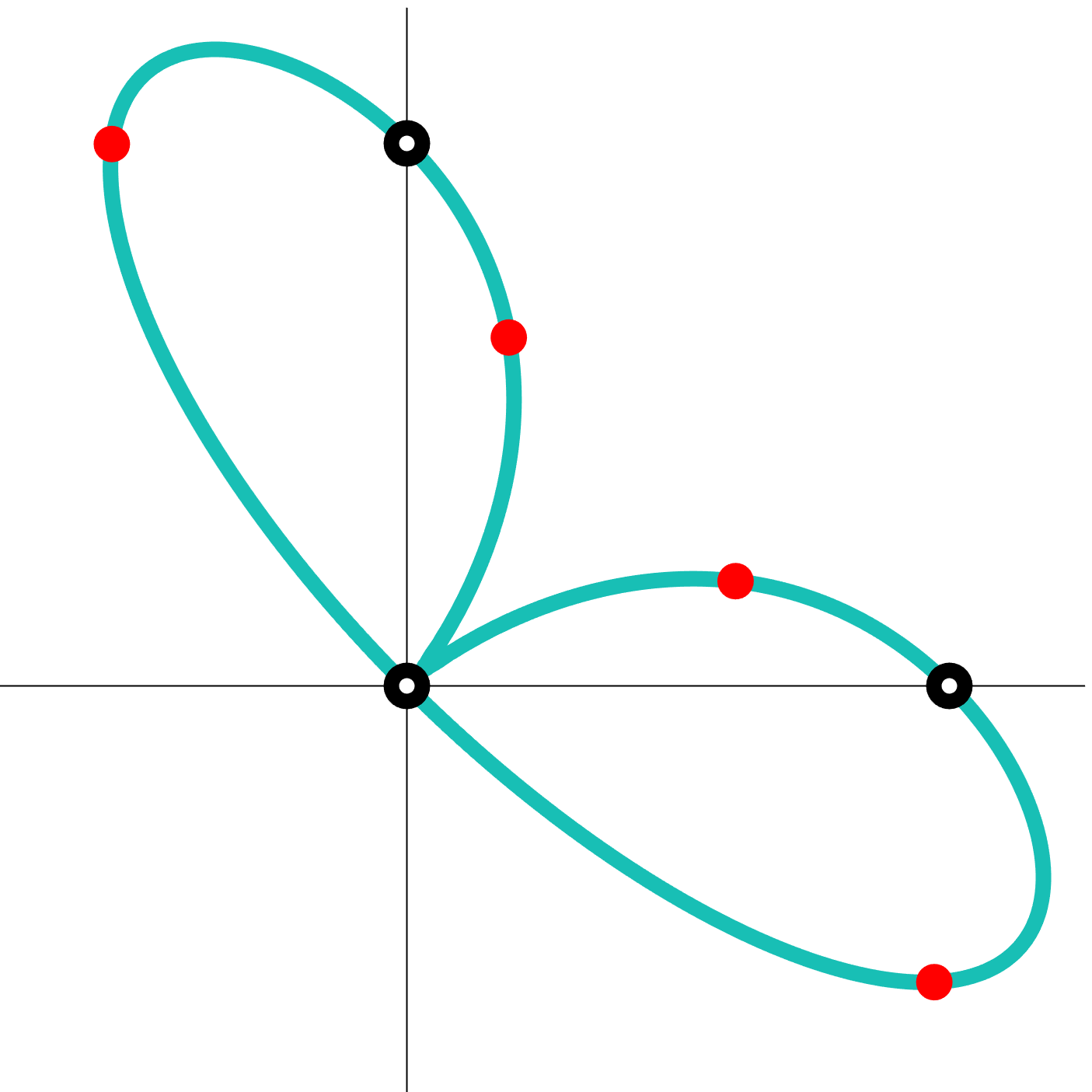}
    \caption{Four routing points for a quartic curve away from the coordinate axes}
    \label{fig:Quartic}
\end{figure}

For another example, consider the compact quartic 
curve $X\subset\bR^2$ from \cite[Ex.~9]{RealTropical}
defined by
$$g(x) = x_1^4 + x_2^4 - (x_1 - x_2)^2(x_1 + x_2) = 0.$$
Since $X\cap(\bR^*)^2\subset\Xsmooth$, one can consider the routing
function 
$$r(x) = \frac{x_1x_2}{((x_1-1/3)^2+(x_2-1/2)^2+1)^2}.$$
This yields four routing points as shown
in Figure~\ref{fig:Quartic}.  
Two of the routing points have a negative coordinate
and are ignored.  This leaves two routing points
with positive coordinates, each of which is a local maximum
yielding that $X\cap(\bR_{>0})^2$ has two smoothly
connected components.

\subsection{Some surfaces}

With the Whitney umbrella used as an illustrative example, 
the following summarizes computing the Euler characteristic 
and the number of smoothly connected components for the following 
surfaces in $\bR^3$:\footnote{See 
\url{https://homepage.univie.ac.at/herwig.hauser/bildergalerie/gallery.html},
\url{https://silviana.org/gallery/hauser/}, and
\url{https://www-sop.inria.fr/galaad/surface/}
for additional information.}
\begin{itemize}
    \item (Dingdong) $g(x) =x_1^2 + x_2^2 - x_3^2 + x_3^3$
    \item (Calypso) $g(x) = x_1^2 + x_2^2x_3 - x_3^2$
    \item (Chubs) $g(x) = x_1^4 + x_2^4 + x_3^4 - (x_1^2 + x_2^2 + x_3^2) + 1/2$
    \item (Twilight) $g(x) = (x_1^2 + x_2^2 - 3)^3 + (x_3^3 - 2)^2$
    \item (Eistute) $g(x) = (x_1^2+x_2^2)^3 - 4x_1^2x_2^2(x_3^2+1)$
    \item (Seepferdchen) $g(x) = x_1^4 - 5x_1^2x_2^3/2 + x_2^6 - (x_1 + x_2^2)x_3^3$
\end{itemize}
Table~\ref{tab:Surfaces} summarizes
the results of the computations when taking $f(x) = \|\nabla_{\bR^3} g(x)\|^2$
and 
$$r(x) = \frac{f(x)}{((x_1-c_1)^2+(x_2-c_2)^2+(x_3-c_3)^2+1)^{\deg g}}
\hbox{~~where~~} 
c = 
\left[\begin{array}{c}
0.7978234324 \\
0.6623073432 \\
0.2347907832 \\
\end{array}\right]
$$
in which $c$ was selected randomly.

\begin{table}[!t]
    \centering
    \begin{tabular}{c|c|c|c|c|c}
         Surface & \# index 0 & \# index 1 & \# index 2 & \multicolumn{1}{c|}{$\chi$} & \# smoothly connected components \\
         \hline
         \hline
         Dingdong & 2 & 2 & 1 & 1 & 2 \\
         \hline
         Calypso & 2 & 2 & 0 & 0 & 2 \\
         \hline 
         Chubs & 32 & 44 & 4 & $-8$ & 8 \\
         \hline
         Twilight & 2 & 1 & 1 & 2 & 2 \\
         \hline 
         Eistute & 7 & 4 & 1 & 4 & 4 \\
         \hline
         Seepferdchen & 7 & 11 & 2 & $-2$ & 1 \\
    \end{tabular}
    \caption{Summary data for some algebraic surfaces in $\bR^3$}
    \label{tab:Surfaces}
\end{table}

\subsection{Connectivity in real projective space}\label{sec:RealProjSpace}

In \cite[Ex.~6.4]{MaximalMumfordCurves}, 
for small $\epsilon>0$,
the following octic curve in $\bP_{\bR}^4$
is shown to consist of six~ovals:
$$
\left[\begin{array}{c}
(x_2+x_3)(x_2+x_3-x_4) + \epsilon(x_1^2+2x_1x_3-x_1x_4+x_3^2-x_3x_4) + \epsilon^2 x_0^2 \\
x_0(x_2+x_3-x_4) + \epsilon(x_0x_3-x_0x_4 + x_1x_3 - x_1x_4 + x_2x_4+x_3x_4 - x_4^2) - \epsilon^2 x_4^2 \\
x_0(x_1+x_3)+\epsilon(x_0x_4+x_1x_2+x_1x_3+x_1x_4+x_2x_3-x_2x_4+x_3^2)
\end{array}\right] = 0.
$$
To verify this, fix $\epsilon=10^{-2}$,
and consider the double cover
on the unit sphere in $\bR^5$ by appending $x_0^2+\cdots+x_4^2-1$ to the system above yielding 
a smooth compact degree 16 curve $X\subset\bR^5$.  
Since~$X$ is smooth and compact with
$X\cap V(x_4) = \emptyset$, 
we can use the routing function
$r(x) = x_4$ with $X = X_r$.  
This yields 40 routing points which arise as 20 pairs
of antipodal points.  Thus, one only needs to consider
the 20 routing points with $x_4>0$
which arise as 10 local maxima (index 0) 
and 10 local minima (index 1).
Using gradient ascent from the local minima,
this yields 6 connected components,
2 with a single local maximum and local minimum
and 4 with two local maxima and local minima,
confirming the results in \cite[Ex.~6.4]{MaximalMumfordCurves}.

To justify our choice of $\epsilon$, we also considered the system 
where $\epsilon$ was a free parameter
and computed the discriminant with respect to $\epsilon$.
This computation showed that the smallest positive root of the discriminant with respect 
to $\epsilon$ to be approximately $0.01438729081$. 
Hence, for any $\epsilon>0$ less than this value, the 
structure of the routing points will be the same
yielding an octic curve with the same real geometry.
In particular, this justifies our choice of $\epsilon = 10^{-2}$ as it is smaller than the smallest positive root of the discriminant. 

\subsection{Counting input modes for a five-bar mechanism}

\begin{figure}[!b]
    \centering
    \includegraphics[scale=0.3]{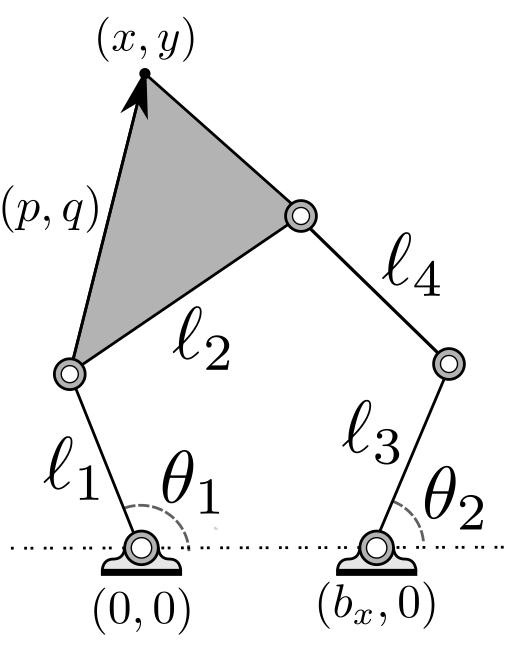}
    \caption{Illustration of a five-bar mechanism}
    \label{fig:FiveBar}
\end{figure}

Input modes of a five-bar mechanism, which is illustrated in Figure~\ref{fig:FiveBar} with more details provided in~\cite[Fig.~2]{OutputModeSwitching}, are the connected
components of the corresponding configuration space 
after removing the input singularities \cite{OutputModeSwitching}.  There are transmission
problems at input singularities causing a loss
of control authority of the end-effector.  
Hence, in \cite{OutputModeSwitching}, a margin around
the input singularities was avoided, which may be
for the application to provide a safety margin,
e.g., to accommodate manufacturing tolerances.  
Nonetheless, the following considers computing the actual
number of input modes without the safety margin.

Adapted from \cite{OutputModeSwitching}, the configuration space
is defined by $g = \{g_1,\dots,g_4\} = 0$ 
with parameters $(b_x,\ell_1,\ell_2,\ell_3,\ell_4,p,q)$
and variables $(x,y,c_1,s_1,c_2,s_2)$
where $c_j = \cos \theta_j$ and $s_j = \sin \theta_j$,
such that 
$$\begin{array}{rcl}
g_1 &=& x^2 + y^2 + \ell_1^2 - 2\ell_1(x c_1+y s_1) - p^2 - q^2, \\
g_2 &=& \ell_2^2 (x^2+y^2) + \ell_1^2 ((\ell_2-p)^2+q^2) + (bx^2 + \ell_3^2 - \ell_4^2)(p^2 + q^2) - 2b_x\ell_2(px + qy) 
\\
& &~-~2\ell_2\ell_3(p(xc_2+ys_2) - q(xs_2-yc_2))
+ 2\ell_1\ell_3((\ell_2p - p^2 - q^2)(c_1c_2 + s_1s_2)) \\
& &~+~2b_x\ell_3(p^2+q^2)c_2 + 2b_x\ell_1((\ell_2p - p^2 - q^2)c_1 + \ell_2qs_1) - 2\ell_1\ell_2\ell_3q(c_1s_2-s_1c_2) \\
& & ~+~2\ell_1\ell_2((p-\ell_2)(xc_1+ys_1)-q(xs_1-yc_1)), \\
g_3 &=& c_1^2 + s_1^2 - 1, \\
g_4 &=& c_2^2 + s_2^2 - 1.
\end{array}
$$
The input singularities are defined by
$$f = \det\left[\begin{array}{cc} \frac{\partial g_1}{\partial x} & \frac{\partial g_1}{\partial y} \\[0.03in]
\frac{\partial g_2}{\partial x} & \frac{\partial g_2}{\partial y} \end{array}\right] = 0.$$
For the parameters, corresponding with
\cite[Ex.~1]{OutputModeSwitching}, 
$$
\begin{array}{rclcrcl}
b_x &=& 0.19882665671846764, & ~~~~ & \ell_1 &=& 0.46540235567944005, \\
p &=& 0.048759206368821334,  & & \ell_2 &=& 0.3486213752206714, \\
q &=& 0.32778886030888477, & & \ell_3 &=& 0.24863642973175545, \\
& & & & \ell_4 &=& 0.4110712177344681,
\end{array}
$$
$X = V_\bR(g)$ is a smooth surface in $\bR^6$ 
and $f$ is a quadratic polynomial 
such that $\{g,f\}=0$ defines two irreducible curves
of degree $6$ in $\bC^6$. 
Using the routing function
$$r = \frac{f}{((x-x_0)^2+(y-y_0)^2+(c_1-c_{10})^2 + (s_1-s_{10})^2 + (c_2-c_{20})^2 + (s_2-s_{20})^2+1)^2}$$
where 
$$
\begin{array}{rclrclrcl}
x_0 &=& 0.919487917032162, & y_0 &=& -0.319228546667734,
& c_{10} &=& 0.170535501959555, \\
s_{10} &=& 0.502534118611306, & c_{20} &=& -0.552376121017726, & s_{20} &=& -0.489809769081462,
\end{array}
$$
were randomly selected, 
there are 8 routing points: four each with $f>0$ and $f<0$.
For $f>0$, there is a maximum and three saddles of index 1.  For $f<0$, there is a minimum and three
saddles of index 1.  Hence, Proposition~\ref{Prop:ConnectedCompRoutingPoints}
immediately yields that $X_r$ has two smoothly connected
components, one with $f>0$ and the other with $f<0$.
When partitioning with a ``thicker kerf,'' 
\cite{OutputModeSwitching} reports $6$ input modes 
with the comment that the ``counts generally do not match the
true number of regions'' due to the safety margin
around the input singularities.

\section{Conclusion}\label{sec:Conclusion}

By using gradient ascent/descent paths on a real algebraic variety,
algorithms were developed for computing the Euler characteristic,
counting the number of smoothly 
connected components, and performing
membership in a smoothly connected component.  
In particular, Algorithm~\ref{alg:Connectivity}
computes a representation of each smoothly
connected component consisting of
routing points and 
gradient ascent/descent paths
which can be used to decide membership
via Algorithm~\ref{alg:ConnectivityQ}.
Such algorithms
could naturally be extended to atomic semi-algebraic sets, with
an example presented for considering the intersection of a real
algebraic variety with the positive orthant.  

As constructed, the algorithms rely upon the ability to construct
a routing function, which relies upon a generic choice of a constant 
vector $c$.  An element of more concern when implementing such a theoretical
algorithm is is the proper tracking of gradient ascent/descent paths emanating
from unstable eigenvector directions.  Future work could be to investigate
using certified differential equation solvers for validating the 
numerically computed trajectories.

\section*{Acknowledgment}\label{sec:Ack}

J.C. and J.D.H. thank Liviu Nicolaescu for discussions involving Morse theory
and corresponding book~\cite{Liviu}.
J.D.H. was supported in part by National Science Foundation grant CCF~2331400
and the Robert and Sara Lumpkins Collegiate Professorship.
H.H. was supported in part by National Science Foundation grants CCF~2331401 and CCF~2212461. C.S. was supported by Simons grant 965262.

\section*{Competing interests}

The authors have no financial or proprietary interests in any material discussed in this article.

\bibliographystyle{abbrv}
\bibliography{ref}

\end{document}